\documentclass[a4paper,11pt]{article}

\usepackage{a4wide}
\usepackage{amsfonts}
\usepackage{amsmath}
\usepackage{amssymb}
\usepackage{amsthm}
\usepackage{graphicx}
\usepackage{latexsym}
\usepackage{graphicx, float, url}

\input amssym.def
\input amssym.tex

\usepackage{color}

\definecolor{VeryLightBlue}{rgb}{0.9,0.9,1}
\definecolor{LightBlue}{rgb}{0.8,0.8,1}
\definecolor{MidBlue}{rgb}{0.5,0.5,1}
\definecolor{DarkBlue}{rgb}{0,0,0.6}
\definecolor{Blue}{rgb}{0,0,1}

\definecolor{Gold}{rgb}{1,0.843,0}

\definecolor{LightGreen}{rgb}{0.88,1,0.88}
\definecolor{MidGreen}{rgb}{0.6,1,0.6}
\definecolor{DarkGreen}{rgb}{0,0.6,0}

\definecolor{VeryLightYellow}{rgb}{1,1,0.9}
\definecolor{LightYellow}{rgb}{1,1,0.6}
\definecolor{MidYellow}{rgb}{1,1,0.5}
\definecolor{DarkYellow}{rgb}{1,1,0.2}

\definecolor{DarkPurple}{rgb}{.6,0,1}

\definecolor{Red}{rgb}{1,0,0}
\definecolor{VeryLightRed}{rgb}{1,0.9,0.9}
\definecolor{LightRed}{rgb}{1,0.8,0.8}
\definecolor{MidRed}{rgb}{1,0.55,0.55}

\long\def\delete#1{}

\def\Om{\Omega}

\def\a{\alpha}
\def\b{\beta}
\def\d{\delta}
\def\g{\gamma}

\def\t{\tau}
\def\ve{\varepsilon}

\def\om{\omega}
\def\z{\zeta}

\def\la{\langle}
\def\ra{\rangle}

\def\mod{{\rm mod~}}

\newcommand{\Z}{\mathbb{Z}}

\newcommand{\Q}{\mathbb{Q}}

\def\Cay{\mathrm{Cay}}

\def\val{{\mathrm{val}}}
\def\SL{\mathrm{SL}}

\def\qed{\hfill$\Box$\vspace{12pt}}

\newcommand{\be}{\begin{equation}}
\newcommand{\ee}{\end{equation}}
\newcommand{\bea}{\begin{eqnarray}}
\newcommand{\eea}{\end{eqnarray}}
\newcommand{\bean}{\begin{eqnarray*}}
\newcommand{\eean}{\end{eqnarray*}}

\newtheorem{thm}{Theorem}[section]
\newtheorem{cor}[thm]{Corollary}
\newtheorem{lem}[thm]{Lemma}

\theoremstyle{definition}
\newtheorem{defn}[thm]{Definition} 
\newtheorem{ex}[thm]{Example} 
\newtheorem{remark}[thm]{Remark}

\newtheorem{problem}[thm]{Problem}

\title{Cyclotomic graphs and perfect codes}
\author{Sanming Zhou \\ \\
School of Mathematics and Statistics\\
The University of Melbourne\\
Parkville, VIC 3010, Australia\\
Email: smzhou@ms.unimelb.edu.au}

\date{May 4, 2018}

\begin{document}
\openup 0.5\jot 
\maketitle

\begin{abstract}
We study two families of cyclotomic graphs and perfect codes in them. They are Cayley graphs on the additive group of $\Z[\z_m]/A$, with connection sets $\{\pm (\z_m^i + A): 0 \le i \le m-1\}$ and $\{\pm (\z_m^i + A): 0 \le i \le \phi(m) - 1\}$, respectively, where $\z_m$ ($m \ge 2$) is an $m$th primitive root of unity, $A$ a nonzero ideal of $\Z[\z_m]$, and $\phi$ Euler's totient function. We call them the $m$th cyclotomic graph and the second kind $m$th cyclotomic graph, and denote them by $G_{m}(A)$ and $G^*_{m}(A)$, respectively. We give a necessary and sufficient condition for $D/A$ to be a perfect $t$-code in $G^*_{m}(A)$ and a necessary condition for $D/A$ to be such a code in $G_{m}(A)$, where $t \ge 1$ is an integer and $D$ an ideal of $\Z[\z_m]$ containing $A$. In the case when $m = 3, 4$, $G_m((\a))$ is known as an Eisenstein-Jacobi and Gaussian networks, respectively, and we obtain necessary conditions for $(\b)/(\a)$ to be a perfect $t$-code in $G_m((\a))$, where $0 \ne \a, \b \in \Z[\z_m]$ with $\b$ dividing $\a$. In the literature such conditions are known to be sufficient when $m=4$ and $m=3$ under an additional condition. We give a classification of all first kind Frobenius circulants of valency $2p$ and prove that they are all $p$th cyclotomic graphs, where $p$ is an odd prime. Such graphs belong to a large family of Cayley graphs that are efficient for routing and gossiping. 
 

\end{abstract}

\section{Introduction}

Perfect codes have been important objects of study ever since the dawn of coding theory in the late 1940s, and after more than six decades they still receive much attention today. Hamming and Golay codes are well known examples of perfect codes, and their importance to information theory has been widely recognised. So far a large number of beautiful results on perfect codes have been produced, as seen in the survey papers \cite{Heden,vanLint}. As generalizations of perfect codes in the classical setting, in \cite{Biggs} Biggs initiated the study of perfect codes in distance-transitive graphs, namely those graphs whose automorphism groups are transitive on the set of ordered pairs of vertices at distance $i$, for every $i$ from $0$ to the diameter of the graph. In the same paper he generalized the celebrated Lloyd's Theorem in the classical setting to distance-transitive graphs. (Lloyd's Theorem asserts that if a perfect $e$-code of length $n$ exists, then the zeros of a certain polynomial of degree $e$ must be distinct integers among $1, 2, \ldots, n$.) Distance-transitive graphs are distance-regular graphs, which in turn can be viewed as association schemes. In \cite{Del} Delsarte pioneered the study of perfect codes in association schemes. Since then a great deal of work has been done in this research direction (see e.g. \cite{Bannai, Del, E}).  
  
The study of perfect codes in general graphs began with \cite{Kra}. A \emph{code} in a graph $X = (V, E)$ is a non-empty subset of $V$. Given an integer $t \ge 1$, the \emph{ball} with radius $t$ and centre $v \in V$ is defined as $B_{t}(v, X) := \{u\in V: d(v, u) \le t\}$, where $d(v, u)$ is the distance in $X$ between $v$ and $u$. A code $C \subseteq V$ is called a \emph{perfect $t$-error-correcting code} or a \emph{perfect $t$-code} in $X$ if the balls $B_{t}(v, X)$ with radius $t$ and centres $v \in C$ form a partition of $V$. In graph theory, $B_{t}(v, X)$ is called the {\em $t$-neighbourhood} of $v$ in $X$, each vertex in $B_{t}(v, X)$ is said to be {\em $t$-dominated} by $v$, a perfect $t$-code in $X$ is called a \emph{perfect $t$-dominating set} of $X$, and a perfect $1$-code in $X$ is called an \emph{efficient dominating set} or \emph{independent perfect dominating set} (see e.g. \cite{DS, Lee, OPR07, KM13}). A $q$-ary perfect $t$-code of length $n$ in the classical setting is simply a perfect $t$-code in the corresponding Hamming graph $H(n, q)$.  

Since $H(n, q)$ is a Cayley graph on $\Z_q^n$, perfect codes in Cayley graphs on finite groups can be thought as another avenue of generalizing perfect codes in the classical setting. Perfect codes in Cayley graphs are also closely related to factorizations and tilings of groups \cite{HBZ18}. So far several results on perfect codes in Cayley graphs have been produced, but the area is still wide open. In \cite{Lee} it was proved that a `normal subset' of a group $G$ is a perfect 1-code in a Cayley graph on $G$ if and only if there exists a covering from the Cayley graph to a complete graph such that $C$ is a fibre of the corresponding covering projection. (In \cite{Lee} a subset $C$ of $G$ is called normal if $gC = Cg$ for any $g \in G$; this is equivalent to saying that $C$ is closed under conjugation.) In \cite{E87} perfect 1-codes in a Cayley graph with connection set closed under conjugation were studied by way of equitable partitions, yielding a nonexistence result in terms of irreducible characters of the underlying group. In \cite{HBZ18} several results about when a normal subgroup of a finite group is a perfect $1$-code in some Cayley graph of the group were obtained. 

In \cite{T04} it was proved that there is no perfect 1-code in any Cayley graph on $\mathrm{SL}(2, 2^f)$, $f > 1$ with respect to any connection set closed under conjugation. In \cite{DS} a methodology for constructing E-chains of Cayley graphs was given and was used to construct infinite families of E-chains of Cayley graphs on symmetric groups, where an \emph{E-chain} is a countable family of nested graphs each containing a perfect 1-code. Perfect 1-codes in circulants were studied in \cite{DSLW16, FHZ16, OPR07, KM13}, and those in directed products of cycles were completely characterized in \cite{Ze}. In \cite{MBG} sufficient conditions for Gaussian and Eisenstein-Jacobi graphs to contain perfect codes were given. Quotients of Gaussian graphs and their applications to constructing perfect codes were further discussed in \cite{MBCSG10}. In \cite{MBG09} a certain Cayley graph defined on the integer quaternions right-modulo a fixed nonzero element was introduced and perfect 1-codes in it were constructed. 
 
In general, it is challenging to construct perfect codes in Cayley graphs -- many Cayley graphs do not contain any perfect code at all. Inspired by \cite{MBG} and our own work \cite{TZ3, Z} on Frobenius graphs, in this paper we study the following two families of Cayley graphs and perfect codes in them. Let $\z_m$ ($m \ge 2$) be an $m$th primitive root of unity, say $\z_m = e^{2\pi i/m}$, and $A$ a nonzero ideal of the ring $\Z[\z_m]$ of algebraic integers in the cyclotomic field $\Q(\z_m)$. We define the {\em $m^{th}$ cyclotomic graph} with respect to $A$, denoted by $G_{m}(A)$, to be the Cayley graph on the additive group of the quotient ring $\Z[\z_m]/A$ with respect to the connection set $\{\pm (\z_m^i + A): 0 \le i \le m-1\}$. We define the {\em second kind $m^{th}$ cyclotomic graph} with respect to $A$, denoted by $G^*_{m}(A)$, to be the Cayley graph on the same group with respect to the connection set $\{\pm (\z_m^i + A): 0 \le i \le \phi(m)-1\}$, where $\phi$ is Euler's function. In the case when $m=3, 4$, $G_{3}((\a))$ and $G_{4}((\a))$ are precisely the Eisenstein-Jacobi and Gaussian networks \cite{FB, MBG, MBSMG}, respectively, where $(\a)$ is the principal ideal generated by $0 \ne \a \in \Z[\z_m]$. These two special families of cyclotomic graphs are closely related to two families of Frobenius circulants as shown in \cite[Lemma 5]{Z1} and \cite[Theorem 5]{TZ3}. 

We prove that the distance in $G^*_{m}(A)$ between two vertices is the Mannheim distance \cite{FG, Huber94} (Lemma \ref{lem:equiv}). Based on this observation we give a necessary and sufficient condition (Lemma \ref{lem:ptdom}) for a subring $D/A$ of $\Z[\z_m]/A$ to be a perfect $t$-code in $G^*_{m}(A)$ (that is, a perfect $t$-code on $\Z[\z_m]/A$ with respect to the Mannheim distance), where $t \ge 1$ and $D$ is an ideal of $\Z[\z_m]$ containing $A$. We also give a necessary condition for $D/A$ to be a perfect $t$-code in $G_{m}(A)$ (Lemma \ref{lem:ptdom}). Applying this result to the case $m=4$, we show that the sufficient condition given in \cite[Theorems 18]{MBSMG} for $(\b)/(\a)$ to be a perfect $t$-code in the Gaussian network $G_4((\a))$ is also necessary (Theorem \ref{thm:gaussian}), where $\a$ and $\b$ are nonzero elements of $\Z[i]$ with $\b$ dividing $\a$. We also give a necessary condition for $(\b)/(\a)$ to be a perfect $t$-code in the Eisenstein-Jacobi network $G_3((\a))$ (Theorem \ref{thm:EJ}), where $\a$ and $\b$ are nonzero elements of $\Z[\rho]$ (where $\rho = (1+\sqrt{-3})/2$) with $\b$ dividing $\a$. It was proved in \cite[Theorem 24]{MBG} that this necessary condition is sufficient in the case when $\a = a + b \rho$ with $\gcd(a, b) = 1$. We show that the condition $\gcd(a, b) = 1$ can be removed. Therefore, in the case when $m=3,4$, all perfect codes in $G_m((\a))$ of the form $(\b)/(\a)$ are known explicitly. An example can be found at the end of the paper. 

As mentioned above, one of the motivations for our work is the study of Frobenius graphs in the context of communication network design. Due to the work in \cite{FLP, S, Z} it is known that first kind Frobenius graphs are `perfect' as far as routing and gossiping are concerned, in the sense that they achieve the smallest possible edge-forwarding and arc-forwarding indices \cite{HMS, Z} and the smallest possible gossiping time \cite{Z} under the store-and-forward, all-port and full-duplex model. (An {\em arc} in a graph is an ordered pair of adjacent vertices.) These features together with the importance of circulants as communication networks \cite{BCH} make it desirable to classify first kind Frobenius circulants. This has been achieved in \cite{TZ} and \cite{TZ3} in the case of valency 4 and 6, respectively. (See also \cite{TZ1, TZ2, Z1} for related results.) In this paper we classify first kind Frobenius circulants of valency $2p$ for any odd prime $p$ (Theorem \ref{thm:2p}), and prove that all of them are $p^{\text{th}}$ cyclotomic graphs (Theorem \ref{cor:cir-cyclo}). Before establishing this connection we prove a few basic properties of cyclotomic graphs (Theorem \ref{thm:basic}). In particular, we prove that $G_{m}(A)$ is arc-transitive and rotational, and thus can be embedded on a closed orientable surface as a balanced regular Cayley map. 

Many problems arise from our study in this paper. One of them is concerned with constructing more perfect codes in cyclotomic graphs, possibly with the help of Lemma \ref{lem:ptdom} and Corollary \ref{cor:ptdom}. See Problem \ref{prob} in the special case where $p$ is an odd prime such that $\Z[\zeta_p]$ is a principal ideal domain.

\section{Notation and definitions}
\label{sec:def}

We follow \cite{Dixon-Mortimer} and \cite{IR, Wash}, respectively, for terminology and notation in group theory and number theory. If $G$ is a group acting on a set $\Om$ and $\a \in \Om$, the {\em stabilizer} of $\a$ in $G$ is the subgroup $G_{\a} := \{g \in G: \a^g = \a\}$ of $G$ and the {\em $G$-orbit} containing $\a$ is $\a^G := \{\a^g: g \in G\}$. If $H$ and $K$ are groups such that $H$ acts on $K$ as a group, the {\em semidirect product} $K \rtimes H$ is the group defined on the set $K \times H$ with operation given by $(x, u)(y, v) := (x y^{u^{-1}}, uv)$ for $(x, u), (y, v) \in K \times H$.

Given a group $G$ and a subset $S$ of $G$ such that $1_G \not \in S = S^{-1} := \{s^{-1}: s \in S\}$ (where $1_G$ is the identity element of $G$), the {\em Cayley graph} on $G$ with respect to $S$, $\Cay(G, S)$, is defined to have vertex set $G$ such that $x, y \in G$ are adjacent if and only if $xy^{-1} \in S$. A {\em complete rotation} \cite{HMP} of $\Cay(G, S)$ is an automorphism of $G$ which fixes $S$ setwise and induces a cyclic permutation on $S$; $\Cay(G, S)$ is {\em rotational} if it admits a complete rotation. A Cayley graph on a cyclic group is called a {\em circulant}. More explicitly, for a subset $S$ of the additive group of ring $\Z_n$ such that $[0] \not \in S = -S := \{-s: s \in S\}$, $\Cay(\Z_n, S)$ is a circulant of order $n$ and valency $|S|$. 

A transitive group $G$ on $\Om$ is called a {\em Frobenius group} \cite{Dixon-Mortimer} if it is not regular but only the identity element can fix two points of $\Om$. It is well known (see e.g.~\cite[Section 3.4]{Dixon-Mortimer}) that a finite Frobenius group $G$ has a nilpotent normal subgroup $K$, called the {\em Frobenius kernel} of $G$, which is regular on $\Om$. Hence $G = K \rtimes H$, where $H$ is the stabilizer of a point of $\Om$. Since $K$ is regular on $\Om$, we may identify $\Om$ with $K$ in such a way that $K$ acts on itself by right multiplication, and we choose $H$ to be the stabilizer of $1_K$ so that $H$ acts on $K$ by conjugation. Thus an $H$-orbit on $K$ is of the form $x^H := \{h^{-1}xh: h \in H\}$, where $x \in K$. A {\em first kind $G$-Frobenius graph} \cite{FLP, Z} is a Cayley graph $X=\Cay(K, S)$ on $K$, where $S = a^H$ for some $a \in K$ satisfying $\la a^H \ra = K$, with $|H|$ even or $a$ an involution. By abusing terminology we say that $X$ is a first kind Frobenius graph with {\em kernel} $K$.  

All graphs in the paper are finite and undirected. A graph $X$ is {\em $k$-valent} if all its vertices have valency $k$; in this case $k = \val(X)$ is called the valency of $X$. A graph $X$ is {\em $G$-vertex-transitive} ({\em $G$-edge-transitive}, {\em $G$-arc-transitive}, respectively) if $G$ is a subgroup of the automorphism group of $X$ that is transitive on the set of vertices (edges, arcs, respectively) of $X$.

\section{Cyclotomic graphs}
\label{sec:cyclo}

In this section we introduce cyclotomic graphs and prove a few basic properties of them. 

It is well known (see e.g. \cite[Theorem 2.6]{Wash}) that the ring of algebraic integers in the cyclotomic field $\Q(\z_m)$ is $\Z[\z_m] := \{a_0 + a_1 \z_m + \ldots + a_{m-1} \z_m^{m-1}: a_0, a_1, \ldots, a_{m-1} \in \Z\}$.
It is also known (see e.g. \cite[Theorem 2.5]{Wash}) that $[\Q(\zeta_m):\Q] = \phi(m)$ with $1, \z_m, \ldots, \z_m^{\phi(m)-1}$ a basis for $\Q(\zeta_m)$ over $\Q$. 

\begin{defn}
\label{def:cyc-graph}
Let $m \ge 2$ be an integer and $A \ne \{0\}$ an ideal of $\Z[\z_m]$. Define 
$$
G_{m}(A) := \Cay(\Z[\z_m]/A, E_m/A)
$$
to be the Cayley graph on the additive group of $\Z[\z_m]/A$ with respect to 
\be
\label{eq:pma}
E_m/A := \{\pm (\z_m^i + A): 0 \le i \le m-1\}.
\ee
Define
$$
G^*_m(A) := \Cay(\Z[\z_m]/A, E^*_m/A)
$$
to be the Cayley graph on the additive group of $\Z[\z_m]/A$ with respect to 
$$
E^*_m/A := \{\pm (\z_m^i + A): 0 \le i \le \phi(m) - 1\}.
$$
We call $G_{m}(A)$ and $G^*_m(A)$ the {\em $m^{th}$ cyclotomic graph} and the \emph{second kind $m^{th}$ cyclotomic graph} with respect to $A$, respectively.  

If $A = (\a) \ne \{0\}$ is a principal ideal of $\Z[\z_m]$, we write $G_{m}(\a)$ and $ G^*_m(\a)$ in place of $G_{m}((\a))$ and $G^*_m((\a))$, respectively.
\end{defn}

In other words, $G_{m}(A)$ ($G^*_m(A)$, respectively) has vertex set $\Z[\z_m]/A$ such that $\a + A, \b + A \in \Z[\z_m]/A$ are adjacent if and only if $\a-\b-\z_m^i \in A$ or $\a-\b+\z_m^i \in A$ for some $i$ with $0 \le i \le m-1$ ($0 \le i \le \phi(m) - 1$, respectively). Of course $G^*_m(A)$ is a spanning subgraph of $G_m(A)$.  

Let us recall a few basic definitions about $\Z[\z_m]$. The \emph{norm} of a nonzero ideal $A$ of $\Z[\z_m]$, $N(A)$, is defined \cite[Chapter 14]{IR} as the cardinality of $\Z[\z_m]/A$. For $\a \in \Q(\z_m)$, let $N(\a)$ denote the usual norm $N_{\Q(\z_m)/\Q}(\a)$ of $\a$ (see e.g. \cite[Chapter 12]{IR}). Since $\Q(\z_m)$ is a cyclotomic number field, $N(\a) \ge 0$ is an integer for any $\a \in \Q(\z_m)$, and $N(\a) = 0$ if and only if $\a = 0$. It is well known that $N((\a)) = |N(\a)|$ (see e.g. \cite[Proposition 14.1.3]{IR}). 

The multiplicative group $(\Z[\z_m]/A)^*$ of units of $\Z[\z_m]/A$ acts on the additive group of $\Z[\z_m]/A$ by right multiplication: $(\a + A)^{\g + A} = (\a + A)(\g + A) = \a \g + A$. This is an action as a group because it respects the addition of $\Z[\z_m]/A$. Thus the semidirect product 
$$
L := (\Z[\z_m]/A) \rtimes (\Z[\z_m]/A)^*
$$ 
is well-defined. It is straightforward to verify that
\be
\label{eq:action}
(\a + A)^{(\b+A, \g + A)} = (\a+\b)\g + A,\;\, \a + A \in \Z[\z_m]/A,\; (\b+A, \g + A) \in L
\ee
defines a faithful action (as a set) of $L$ on $\Z[\z_m]/A$.  

The subset $E_m/A$ of $(\Z[\z_m]/A)^*$ is a cyclic subgroup of $(\Z[\z_m]/A)^*$, which we denote by $H_A$. We have 
\be
\label{eq:HA}
H_A = \begin{cases}
\la (-\z_m) + A \ra, & \text{ if $m$ is odd}\\
\la \z_m + A \ra, & \text{ if $m$ is even. }
\end{cases}
\ee
In fact, if $m$ is even, then $H_A = \la (-\z_m) + A, (-1)+A \ra = \la \z_m + A, (-1)+A \ra = \la \z_m + A \ra$ as $\z_m^{m/2} = -1$. 

Given a generating set $S$ of $G$ and a cyclic permutation $\rho$ of $S$, a {\em Cayley map} \cite{JS, SS} is a 2-cell embedding of $\Cay(G, S)$ on an orientable surface such that for each vertex $g \in G$, the cyclic permutation of the arcs $(g, sg)$, $s \in S$ induced by a fixed orientation of the surface coincides with $\rho$. A Cayley map is {\em balanced} \cite{SS} if $\rho(s^{-1}) = \rho(s)^{-1}$ for every $s \in S$, and {\em regular} if its automorphism group is regular on the set of arcs of $\Cay(G, S)$. It is known that the existence of a complete rotation in a Cayley graph is equivalent to the existence of a 2-cell embedding of the graph on a closed orientable surface as a balanced regular Cayley map. 

\begin{thm}
\label{thm:basic}
Let $A \ne \{0\}$ be an ideal of $\Z[\z_m]$, where $m \ge 2$, and let $H_A$ be as in \eqref{eq:HA}. Then the following hold:
\begin{itemize}
\item[\rm (a)] $G_{m}(A)$ is a finite, connected, undirected graph of order $N(A)$ and valency $\val(G_{m}(A))$ a divisor of $2m$; moreover, $\val(G_{m}(A)) = 2m$ if and only if $1 \pm \z_m^i \not \in A$ for $1 \le i \le m-1$;
\item[\rm (b)] under the assumption that $ 2 \not \in A$ and $m$ is odd, if there exists $i \ge 1$ such that $1-\z_m^{i} \in A$, say $d \ge 1$ is the smallest integer with this property, then $\val(G_{m}(A)) = 2d$ or $d$, depending on whether $d$ is odd or even with $1 + \z_m^{d/2} \in A$; if there exists $i \ge 1$ such that $1+\z_m^{i} \in A$, then the smallest integer $d$ with this property must be even and $\val(G_{m}(A)) = 2d$;
\item[\rm (c)] if $ 2 \not \in A$ and $m$ is even, then the smallest positive integer $d$ such that $1 - \z_m^d \in A$ must be even and $\val(G_{m}(A)) = d$;
\item[\rm (d)] $G_{m}(A)$ admits $(\Z[\z_m]/A) \rtimes H_{A}$ as a group of automorphisms acting faithfully and transitively on the vertex set and regularly on the arc set; 
\item[\rm (e)] $G_{m}(A)$ is a rotational Cayley graph and hence can be 2-cell-embedded on a closed orientable surface as a balanced regular Cayley map.
\end{itemize}
\end{thm}

\begin{proof}
(a) By \cite[Proposition 12.2.3]{IR}, $N(A)$ is finite, that is, $G_{m}(A)$ is a finite graph with $N(A)$ vertices. Since $H_A$ is closed under taking negative elements, $G_{m}(A)$ is an undirected graph. Since by \eqref{eq:HA}, $H_A$ is a cyclic subgroup of $(\Z[\z_m]/A)^*$ and $(-\z_m)^{2m} + A = \z_m^{2m} + A = 1+ A$, the order of $H_A$ (that is, $\val(G_{m}(A))$) is a divisor of $2m$. By the definition of $G_{m}(A)$, there is at least one path in $G_{m}(A)$ from $A$ to any $\a + A \in \Z[\z_m]/A$. (For example, if $\a = 2-\z_m + 2\z_m^3$, then the sequence $A, 1+A, 2+A, (2-\z_m)+A, (2-\z_m+\z_m^3)+A, (2-\z_m+2\z_m^3)+A$ gives a path from $A$ to $\a + A$.) Therefore, $G_{m}(A)$ is connected. It is clear that $\val(G_{m}(A)) = 2m$ if and only if $\z_m^i \pm \z_m^j \not \in A$ for $0 \le i < j \le m-1$, or equivalently $1 \pm \z_m^j \not \in A$ for $1 \le j \le m-1$.

(b) Suppose $2 \not \in A$ and $m$ is odd. Then $H_A = \la (-\z_m) + A \ra$. 

\smallskip
\textsf{Case 1:}~There exists an integer $i \ge 1$ such that $\z_m^i + A = 1 + A$. Let $d$ be the smallest integer with this property. If $d$ is even, then $H_A = \{1+A, -\z_m + A, \ldots, \z_m^{d-2} + A, -\z_m^{d-1} + A\}$. Since $-1 + A \in H_A$, we have $-1 + A = -\z_m^{2i+1} + A$ for some $1 \le 2i+1 \le d - 1$, or $-1 + A = \z_m^{2i} + A$ for some $1 \le 2i \le d-2$ (note that $i \ne 0$ as $2 \not \in A$). In the former case we obtain $1 - \z_m^{2i+1} \in A$, which contradicts the assumption that $d$ is the smallest positive integer such that $1 - \z_m^{d} \in A$. In the latter case, $1 + \z_m^{2i} \in A$ and so $(1 + \z_m^{2i}) - (1 - \z_m^{d}) \in A$. This gives $1 + \z_m^{d-2i} \in A$ and hence $(1 + \z_m^{d-2i}) - (1 + \z_m^{2i}) = \z_m^{d-2i} - \z_m^{2i} \in A$, which contradicts the choice of $d$ unless $d - 2i = 2i$. (In fact, if $d - 2i > 2i$, then by $\z_m^{d-2i} - \z_m^{2i} \in A$ we have $1 - \z_m^{d-4i} \in A$, contradicting the minimality of $d$. If $d - 2i < 2i$, then $1 - \z_m^{4i-d} \in A$, which again is a contradiction as $0 < 4i - d \le 2(d-2)-d < d$.) In the case when $d - 2i = 2i$, we have $1 + \z_m^{d/2} \in A$, $H_A = \{\pm(1+A), \pm(-\z_m + A), \ldots, \pm(\z_m^{2i-2} + A), \pm(-\z_m^{2i-1} + A)\}$, and $G_{m}(A)$ has valency $d$. Assume $d$ is odd. Then $H_A = \{\pm(1+A), \pm(-\z_m + A), \ldots, \pm(-\z_m^{d-2} + A), \pm(\z_m^{d-1} + A)\}$. By the minimality of $d$ we have $1+A \ne \z_m^i + A$ for $1 \le i \le d-1$. We have also $1+A \ne -1+A$ as $2 \not \in A$ by our assumption. If $1+A = -\z_m^i + A$ for some $1 \le i \le d-1$, then $1 + \z_m^i \in A$ and so $\z_m^i + \z_m^d = (1 + \z_m^i)-(1 - \z_m^d) \in A$. Thus $1+\z_m^{d-i} \in A$ and therefore $\z_m^i - \z_m^{d-i} \in A$, which contradicts the choice of $d$ as $i \ne d-i$ due to $d$ being odd. 

In summary, we have proved that in Case 1 either (i) $d$ is even, $1 + \z_m^{d/2} \in A$ and $G_{m}(A)$ has valency $d$, or (ii) $d$ is odd and $G_{m}(A)$ has valency $2d$. 

\smallskip
\textsf{Case 2:}~There exists an integer $i \ge 1$ such that $-\z_m^i + A = 1 + A$. Let $d$ be the smallest integer with this property. If $d$ is even, then $H_A = \{\pm(1+A), \pm(-\z_m + A), \ldots, \pm(-\z_m^{d-2} + A), \pm(\z_m^{d-1} + A)\}$. Similar to the proof above, one can show that $G_{m}(A)$ has valency $2d$. If $d$ is odd, then $H_A = \{1+A, -\z_m + A, \ldots, -\z_m^{d-2} + A, \z_m^{d-1} + A\}$. Since $-1 + A \in H_A$, we have $-1 + A = -\z_m^{2i+1} + A$ for some $1 \le 2i+1 \le d - 2$, or $-1 + A = \z_m^{2i} + A$ for some $1 \le 2i \le d-1$ (note that $i \ne 0$ as $2 \not \in A$). The latter case cannot happen as it contradicts the minimality of $d$. In the former case we obtain $\z_m^{2i+1} + \z_m^d \in A$ and so $1 + \z_m^{d - 2i -1} \in A$, which also contradicts the minimality of $d$. 

In summary, in Case 2, $d$ is even and $G_{m}(A)$ has valency $2d$.

We claim that Cases 1 and 2 coexist if and only if (i) in Case 1 occurs. In fact, let $d_1, d_2 \ge 1$ be the smallest integers such that $1-\z_m^{d_1}, 1+\z_m^{d_2} \in A$. Then $\z_m^{d_1}+\z_m^{d_2} \in A$ and so $1+\z_m^{d_2 - d_1}, 1+\z_m^{d_1 - d_2} \in A$. By the minimality of $d_2$, we have $d_1  \ge 2d_2$. We also have $\z_m^{d_2} - \z_m^{d_1 - d_2} = (1+\z_m^{d_2}) - (1+\z_m^{d_1 - d_2}) \in A$ and so $1 - \z_m^{d_1 - 2d_2} \in A$. By the minimality of $d_1$, we have $d_1 = 2d_2$, yielding (i) in Case 1.

(c) Suppose $2 \not \in A$ and $m$ is even. Then $H_A = \la \z_m + A \ra$. Let $d$ be the smallest positive integer such that $\z_m^d + A = 1 + A$. (The existence of $d$ is ensured by the fact that $\z_m^m + A = 1 + A$.) Then $H_A = \{1+A, \z_m + A, \ldots, \z_m^{d-2} + A, \z_m^{d-1} + A\}$ and $-1+A = \z_m^{i} + A$ for some $1 \le i \le d-1$ (note that $i \ne 0$ as $2 \not \in A$). So $\z_m^{i} + \z_m^d \in A$ and $1 + \z_m^{d - i} \in A$, which together with $1 + \z_m^{i} \in A$ implies $\z_m^{i} - \z_m^{d - i} \in A$. This together with the minimality of $d$ implies that $d = 2i$ and hence $H_A = \{\pm(1+A), \pm(\z_m + A), \ldots, \pm(\z_m^{i-1} + A)\}$. It follows that $G_{m}(A)$ has valency $d$. 

(d) Since $H_A$ is a subgroup of $(\Z[\z_m]/A)^*$, the semidirect product 
$$
H:=(\Z[\z_m]/A) \rtimes H_{A}
$$ 
is a well-defined subgroup of $L$. Since $L$ is faithful on $\Z[\z_m]/A$, so is $H$. We claim that $H$ preserves the adjacency and non-adjacency relations of $G_{m}(A)$. In fact, the images of $\a_1 + A, \a_2 + A \in \Z[\z_m]/A$ under $(\b+A, \pm (\z_m^i + A)) \in H$ are $\pm ((\a_1 + \b)\z_m^i + A)$ and $\pm ((\a_2 + \b) \z_m^i + A)$, respectively. Since the difference between these two elements is $\pm ((\a_1 - \a_2) \z_m^i + A)$, it follows from the definition of $G_{m}(A)$ that $\a_1 + A$ and $\a_2 + A$ are adjacent in $G_{m}(A)$ if and only if their images under $(\b+A, \pm (\z_m^i + A))$ are adjacent in $G_{m}(A)$. Therefore, $H$ respects the adjacency and non-adjacency relations of $G_{m}(A)$. Thus $G_{m}(A)$ admits $H$ as a group of automorphisms acting faithfully on the vertex set. The subgroup $\Z[\z_m]/A$ of $H$ is transitive on $\Z[\z_m]/A$ by addition, and so $H$ is transitive on the vertex set of $G_{m}(A)$. In view of (\ref{eq:action}), the image of $A \in \Z[\z_m]/A$ under $(\b+A, \pm (\z_m^i + A))$ is $\pm (\b \z_m^i + A)$, which is equal to $A$ if and only if $\b \in A$. Thus the stabilizer of the vertex $A$ of $G_{m}(A)$ under the action of $H$ is the subgroup 
$$
H^*_A := \{(A, \pm (\z_m^i + A)): 0 \le i \le m-1\}
$$ 
of $H$, which is isomorphic to $H_A$. It follows that this stabilizer is transitive on the neighbourhood $E_m/A$ of $A$ in $G_{m}(A)$, because $(\ve \z_m^i + A)^{(\b+A, \ve' \z_m^j + A)} = \pm (\z_m^{i+j} + A)$ (where $\ve, \ve' = \pm 1$) by (\ref{eq:action}). This together with the vertex-transitivity of $H$ on $\Z[\z_m]/A$ implies that $H$ is transitive on the arc set of $G_{m}(A)$. Moreover, by the orbit-stabiliser lemma, the order of $H$ is equal to $N(A) \cdot |H^*_A| = N(A) \cdot \val(G_{m}(A))$, which is the number of arcs of $G_{m}(A)$. Therefore, $H$ must be regular on the arc set of $G_{m}(A)$.    

(e) In the case when $m$ is odd, $(A, (-\z_m) + A) \in H^*_A$ generates the cyclic group $H^*_A$, fixes setwise the neighbourhood $E_m/A = \{(-\z_m)^i + A: 0 \le i \le 2m-1\}$ of $A$, and permutes the neighbours of $A$ in a cyclic manner by  $((-\z_m)^i + A)^{(A, (-\z_m) + A)} = (-\z_m)^{i+1} + A$, $0 \le i \le 2m-1$. Therefore, $(A, (-\z_m) + A)$ is a complete rotation of $G_{m}(A)$ and hence $G_{m}(A)$ can be 2-cell-embedded on a closed orientable surface as a balanced regular Cayley map. Similarly, when $m$ is even, the same result holds with $(A, \z_m + A)$ a complete rotation of $G_{m}(A)$. 
\qed
\end{proof}

\begin{remark}
\label{rem:rho}
Choose $\z_3 = -\rho := -(1+\sqrt{-3})/2$ so that $\rho^2 - \rho + 1 = 0$. Then $\Z[\rho] = \{x+y\rho: x, y \in \Z\}$ is the ring of {\em Eisenstein-Jacobi integers} with norm defined by $N(x+y\rho) = x^2 + xy + y^2$. It is known that $\Z[\rho]^* = \la \rho \ra = \{\pm \rho^i: i=0,1,2\}$. Since $\Z[\rho]$ is an Euclidean domain, every nonzero ideal of it is a principal ideal $(\a)$, and $G_{3}(\a)$ is precisely the {\em Eisenstein-Jacobi (EJ) graph} $EJ_{\a}$ \cite{MBG}. (Unlike \cite[Definition 19]{MBG}, we do not require $\gcd(a, b) = 1$ in $EJ_{a + b\rho}$. In \cite{MBG} the EJ graph $EJ_{a+b\om}$ was defined as the Cayley graph on (the additive group of) $\Z[\omega]/(a+b\om)$ with respect to $\{\pm (1+(a+b\om)), \pm (\om+(a+b\om)), \pm (\om^2 + (a+b\om))\}$, where $\omega = (-1+\sqrt{-3})/2$. As noted in \cite{FB}, although $EJ_{a+b\om}$ has $a^2 - ab + b^2$ vertices and is different from $EJ_{a+b\rho}$, the family of EJ graphs is the same no matter whether $\Z[\rho]$ or $\Z[\omega]$ is used.)

In the case when $m=4$, we choose $\z_4 = i$ (the imaginary unit) and so $\Z[\z_4]$ is the ring of Gaussian integers $\Z[i]$ with norm defined by $N(x+yi) = x^2 + y^2$. Let $0 \ne \a = a+bi \in \Z[i]$ be such that $N(\a) \ge 5$. Then $G_{4}(\a) = G_{4}^*(\a)$ is exactly the {\em Gaussian network} $G_{\a}$ introduced in \cite{MBG}.
\end{remark}

\section{Perfect codes in cyclotomic graphs}
\label{sec:dom}

Let $\a = \sum_{i=0}^{\phi(m)-1} a_{i} \z_m^{i} \in \Z[\z_m]$, where $a_i \in \Z$. Define \cite{FG}
$$
|\a|:= \sum_{i=0}^{\phi(m)-1} |a_{i}|,
$$
where $|a_i|$ denotes the usual absolute value of $a_i$. This defines an integer-valued weight function on $\Z[\z_m]$, called the \emph{Manhattan weight} \cite{Huber94} on $\Z[\z_m]$. This weight then defines the \emph{Manhattan distance} $|\a-\b|$ between $\a$ and $\b$; this is a distance function because it is nonnegative, symmetric and satisfies the triangle inequality (see \cite{FG}).

Denote $\a + A$ by $\bar{\a}$ for $\a \in \Z[\z_m]$ (and in particular $\bar{0} = 0 + A$). Define 
\be
\label{eq:Anorm}
\|\bar{\a}\| := \min\{|\a - \d|: \d \in A\}.
\ee
Note that both $\bar{\a}$ and $\|\bar{\a}\|$ rely on $A$. Note also that $\|\bar{\a}\|$ is independent of the choice of the representative $\a$ in $\a+A$. It was proved in \cite[Section II-D]{FG} that this defines an integer-valued weight on $\Z[\z_m]/A$, that is, (i) $\|\bar{\a}\| \ge 0$ with equality if and only if $\bar{\a} = \bar{0}$, (ii) $\|\bar{\a}\| = \|-\bar{\a}\|$, and (iii) $\|\bar{\a}+\bar{\b}\| \le \|\bar{\a}\|+\|\bar{\b}\|$. This weight then defines the distance $\|\bar{\a} - \bar{\b}\|$ between $\bar{\a}$ and $\bar{\b}$, called the \emph{Mannheim distance} \cite{FG, Huber94}. (This notion was introduced in \cite{FG} when $A$ is a prime ideal, but it works well for any nonzero ideal $A$ of $\Z[\z_m]$.) 

We now show that the Mannheim distance gives the distance between vertices in $G^*_m(A)$. This observation is crucial for us to understand perfect codes in $G^*_m(A)$. 

\begin{lem}
\label{lem:equiv}
Let $m \ge 2$ be an integer and $A$ a nonzero proper ideal of $\Z[\z_m]$. Then for any $\bar{\a}, \bar{\b} \in \Z[\z_m]/A$ the distance in $G^*_m(A)$ between $\bar{\a}$ and $\bar{\b}$ is equal to $\|\bar{\a} - \bar{\b}\|$.
\end{lem}

\begin{proof}
Since $A \ne \Z[\z_m]$, we have $\z_m^{i} \not \in A$ for each integer $i$. 

We first show that $\bar{\a}$ and $\bar{\b}$ are adjacent in $G^*_m(A)$ if and only if $\|\bar{\a} - \bar{\b}\| = 1$. In fact, if $\bar{\a}$ and $\bar{\b}$ are adjacent in $G^*_m(A)$, then $(\a - \b) \pm \z_m^i \in A$ for some $0 \le i \le \phi(m) - 1$, and hence $\|\bar{\a} - \bar{\b}\| = \min\{|\z_m^i - \d|: \d \in A\} \le |\z_m^i| = 1$. However, $\|\bar{\a} - \bar{\b}\| \ge 1$ as $\bar{\a} \ne \bar{\b}$. Therefore, $\|\bar{\a} - \bar{\b}\| = 1$. Conversely, if $\|\bar{\a} - \bar{\b}\| = 1$, then there exists $\d \in A$ such that $|(\a-\b)-\d| = 1$ and so $(\a-\b)-\d = \pm \z_m^i$ for some $0 \le i \le \phi(m) - 1$, implying that $\bar{\a}$ and $\bar{\b}$ are adjacent in $G^*_m(A)$. 

In general, let $s=\|\bar{\a} - \bar{\b}\|$ and let $t = d(\bar{\a}, \bar{\b})$ be the distance between $\bar{\a}$ and $\bar{\b}$ in $G^*_m(A)$. Let $\bar{\a} = \bar{\a}_0, \bar{\a}_1, \ldots, \bar{\a}_t = \bar{\b}$ be a shortest path in $G^*_m(A)$. Since $\bar{\a}_i$ and $\bar{\a}_{i+1}$ are adjacent in $G^*_m(A)$, we have $\|\bar{\a}_{i}-\bar{\a}_{i+1}\| = 1$ by what we proved in the previous paragraph. Since $\bar{\a} - \bar{\b} = \sum_{i=0}^{t-1} (\bar{\a}_{i}-\bar{\a}_{i+1})$, we obtain $s = \|\bar{\a} - \bar{\b}\| \le t$ by the triangular inequality.

In view of (\ref{eq:Anorm}), there exists $\d \in A$ such that $s = |(\a-\b)-\d|$. So we may write $(\a-\b)-\d = c_{i_1} \z_m^{i_1} + \cdots + c_{i_k} \z_m^{i_k} - (d_{j_1} \z_m^{j_1} + \cdots + d_{j_l} \z_m^{j_l})$, where $i_1, \ldots, i_k, j_1, \ldots, j_l$ are pairwise distinct integers between $0$ and $\phi(m)-1$ and $c_{i_1}, \ldots, c_{i_k}, d_{j_1}, \ldots, d_{j_l}$ are positive integers summing up to $s$. Thus the sequence  
$$
A, \z_m^{i_1}+A, \ldots, c_{i_1} \z_m^{i_1}+A, (c_{i_1} \z_m^{i_1}+\z_m^{i_2})+A, \ldots, (c_{i_1} \z_m^{i_1}+c_{i_2} \z_m^{i_2})+A, \ldots, ((\a-\b)-\d)+A
$$
is a path in $G^*_m(A)$ with length $c_{i_1} + \cdots + c_{i_k} + d_{j_1} + \cdots + d_{j_l} = s$. (In each step of the sequence there is an increase or decrease by some $\z_m^{i_r}$.) Since $((\a-\b)-\d)+A = (\a+A)-(\b+A)$ (as $\d \in A$), this sequence gives a path in $G^*_m(A)$ between $\bar{0}$ and $\bar{\a}-\bar{\b}$ and hence $d(\bar{0}, \bar{\a}-\bar{\b}) \le s$. However, we have $d(\bar{\a}, \bar{\b}) = d(\bar{0}, \bar{\a}-\bar{\b})$ because the additive group of $\Z[\z_m]/A$ is regular on the vertex set $\Z[\z_m]/A$ of $G^*_m(A)$ by addition as a group of automorphisms. Therefore, $t \le s$ and the proof is complete. 
\qed 
\end{proof}

Denote by $B_{t}(\bar{\b})$ and $B_{t}^*(\bar{\b})$ the $t$-neighbourhood of $\bar{\b} \in \Z[\z_m]/A$ in $G_m(A)$ and $G^*_m(A)$, respectively. Since $G_m(A)$ and $G^*_m(A)$ are both vertex-transitive, we have $|B_{t}(\bar{\b})| = |B_{t}(\bar{0})|$ and $|B_{t}^*(\bar{\b})| = |B_{t}^*(\bar{0})|$ for all $\bar{\b} \in \Z[\z_m]/A$. By Lemma  \ref{lem:equiv}, 
$$
B_{t}^*(\bar{\b}) = \{\bar{\g} \in \Z[\z_m]/A: \|\bar{\b} - \bar{\g}\| \le t\}.
$$ 
Note that, if $D$ is an ideal of $\Z[\z_m]$ containing $A$, then $D/A = \{\b + A: \b \in D\}$ is a subring of $\Z[\z_m]/A$. The following result easily follows from Lemma \ref{lem:equiv} and the definition of a perfect code. 

\begin{lem}
\label{lem:ptdom}
Let $m \ge 2$ and $t \ge 1$ be integers, and let $A$ and $D$ be nonzero ideals of $\Z[\z_m]$ such that $A \subseteq D$. Then the following hold:
\begin{itemize}
\item[\rm (a)] 
$D/A$ is a perfect $t$-code in $G^*_m(A)$ if and only if 
$$
|B_{t}^*(\bar{0})| = N(D) 
$$
and 
\be
\label{eq:t}
|\eta - \d| \ge 2t+1
\ee
for any $\d \in A$ and $\eta \in D - A$;
\item[\rm (b)] 
$D/A$ is a perfect $t$-code in $G_m(A)$ only when
$$
|B_{t}(\bar{0})| = N(D) 
$$
and (\ref{eq:t}) holds for any $\d \in A$ and $\eta \in D - A$.
\end{itemize}
\end{lem}

\begin{proof}
By Lemma \ref{lem:equiv}, we have: $B_{t}^*(\bar{\b}) \cap B_{t}^*(\bar{\g}) = \emptyset$ for distinct $\bar{\b}, \bar{\g} \in D/A$ $\Leftrightarrow$ $\|\bar{\b} - \bar{\g}\| \ge 2t + 1$ for distinct $\bar{\b}, \bar{\g} \in D/A$ $\Leftrightarrow$ $\|\bar{\eta}\| \ge 2t + 1$ for any $\bar{0} \ne \bar{\eta} \in D/A$ $\Leftrightarrow$ $|\eta - \d| \ge 2t+1$ for any $\d \in A$ and $\eta \in D - A$. We have $|D/A| = N(A)/N(D)$ as $(\Z[\z_m]/A)/(D/A) \cong \Z[\z_m]/D$.  

Using the facts above, we have: $D/A$ is a perfect $t$-code in $G^*_m(A)$ $\Leftrightarrow$ $\{B_{t}^*(\bar{\b}): \bar{\b} \in D/A\}$ is a partition of $\Z[\z_m]/A$ $\Leftrightarrow$ $|D/A| \cdot |B_{t}^*(\bar{\b})| = N(A)$ and $B_{t}^*(\bar{\b}) \cap B_{t}^*(\bar{\g}) = \emptyset$ for distinct $\bar{\b}, \bar{\g} \in D/A$ $\Leftrightarrow$ $|B_{t}^*(\bar{0})| = N(D)$ and \eqref{eq:t}  holds for any $\d \in A$ and $\eta \in D - A$. 

Since $G^*_m(A)$ is a spanning subgraph of $G_m(A)$, we have $B_{t}^*(\bar{\b}) \subseteq B_{t}(\bar{\b})$ for any $\bar{\b} \in \Z[\z_m]/A$. Thus we have: $D/A$ is a perfect $t$-code in $G_m(A)$ $\Leftrightarrow$ $\{B_{t}(\bar{\b}): \bar{\b} \in D/A\}$ is a partition of $\Z[\z_m]/A$ $\Leftrightarrow$ $|D/A| \cdot |B_{t}(\bar{\b})| = N(A)$ and $B_{t}(\bar{\b}) \cap B_{t}(\bar{\g}) = \emptyset$ for distinct $\bar{\b}, \bar{\g} \in D/A$ $\Rightarrow$ $|B_{t}(\bar{0})| = N(D)$ and $B_{t}^*(\bar{\b}) \cap B_{t}^*(\bar{\g}) = \emptyset$ for distinct $\bar{\b}, \bar{\g} \in D/A$ $\Rightarrow$ $|B_{t}(\bar{0})| = N(D)$ and \eqref{eq:t}  holds for any $\d \in A$ and $\eta \in D - A$.  
\qed
\end{proof}

The following is an immediate consequence of Lemma \ref{lem:ptdom}.

\begin{cor}
\label{cor:ptdom}
Let $m \ge 2$ and $t \ge 1$ be integers, and let $0 \ne \a, \b \in \Z[\z_m]$ be such that $\b$ divides $\a$. Then the following hold:
\begin{itemize}
\item[\rm (a)] $(\b)/(\a)$ is a perfect $t$-code in $G^*_m(\a)$ if and only if 
\begin{equation}
\label{eq:deg1}
|B_{t}^*(\bar{0})| = N(\b)
\end{equation} 
and 
\begin{equation}
\label{eq:t1}
|\tau \b| \ge 2t+1
\end{equation} 
for any nonzero $\t \in \Z[\z_m]$;
\item[\rm (b)] $(\b)/(\a)$ is a perfect $t$-code in $G_m(\a)$ only when
\begin{equation}
\label{eq:deg2}
|B_{t}(\bar{0})| = N(\b)
\end{equation} 
and (\ref{eq:t1}) holds any nonzero $\t \in \Z[\z_m]$. 
\end{itemize}
\end{cor}

\begin{remark}
Lemma \ref{lem:ptdom} and Corollary \ref{cor:ptdom} only provide necessary conditions for $D/A$ and $(\b)/(\a)$ to be a perfect $t$-code in $G_m(A)$ and $G_m(\a)$, respectively. We are unable to tell whether these conditions are sufficient due to lack of knowledge of the distance in $G_m(A)$ and $G_m(\a)$. In general, this distance is not the Mannheim distance as observed in \cite{FB, MBG} for $m = 3, 4$. 
\end{remark}

In the special case when $m=3,4$, by using Corollary \ref{cor:ptdom} and the knowledge of the distance in $G_m(A)$ \cite{FB, MBSMG}, we now prove that two sufficient conditions given in \cite{MBG, MBSMG} are also necessary.

As mentioned in Remark \ref{rem:rho}, for $0 \ne \a = a+bi \in \Z[i]$ with $N(\a) \ge 5$, $G_{4}(\a)$ is the Gaussian network $G_{\a}$ introduced in \cite{MBG}. It can be easily seen that 
$G_{\a} \cong G_{i\a^j}$ for any integer $j$. So we may assume $a, b \ge 0$ in subsequent discussion about Gaussian networks. The size of the ball $B_{t}(\bar{0})$ of radius $t$ around $\bar{0} = 0 + (\a)$ in $G_{\a}$ was determined in \cite{MBSMG} for any $t \ge 0$. In particular, it was proved in \cite[Theorems 10-11]{MBSMG} that, if $0 \le a \le b$, then 
\be
\label{eq:4t}
|B_{t}(\bar{0})|=4t,\;\, 0 \le t \le \lfloor (a+b-1)/2 \rfloor. 
\ee
In fact, this formula is also valid when $a > b \ge 0$ as $G_{a+bi} \cong G_{b+ai}$ (see \cite[Section IV]{FB}). It is known that two elements $\b, \g$ of $\Z[i]$ are associates of each other if and only if $\b = i^j \g$ for some integer $j$. Since $\Z[i]$ is a principal ideal domain and associates generate the same principal ideal, any nonzero ideal of $\Z[i]$ is of the form $(c+di)$ or $(c-di)$ for some $c, d \ge 0$.

The `if' part of the following result was proved in \cite[Theorem 18]{MBSMG}, which improved an earlier version \cite[Theorem 14]{MBG} that required the extra condition $\gcd(a, b)=1$. We complete the picture by proving the `only if' part by using Corollary \ref{cor:ptdom}(b). (Note that the condition $t \le \lfloor (a+b-1)/2 \rfloor$ is needed for otherwise $(\b)/(\a)$ may not be a perfect $t$-code in $G_{\a}$.)

\begin{thm}
\label{thm:gaussian}
Let $0 \ne \a = a+bi \in \Z[i]$ (where $a, b \ge 0$), and let $0 \ne \b \in \Z[i]$ be such that $N(\a) \ge 5$ and $\b$ divides $\a$. Let $t$ be an integer between $1$ and $\lfloor (a+b-1)/2 \rfloor$. Then $(\b)/(\a)$ is a perfect $t$-code in $G_{\a}$ if and only if $\b$ is an associate of $t + (t+1)i$ or $t - (t+1)i$. 
\end{thm}

\begin{proof}
We only need to prove the necessity. As mentioned above, we may assume $0 \ne \b = c \pm di \in \Z[i]$, where $c, d \ge 0$. Since $\b$ divides $\a$, we have $(\a) \subseteq (\b)$, $\a = \g \b$ for some $\g \in \Z[i]$, and $N(\b) = c^2 + d^2$ divides $N(\a) = a^2 + b^2$. Since $N(\a) \ge 5$, $G_{\a}$ has valency $4$ by Theorem \ref{thm:basic}(c). Since $1 \le t \le \lfloor (a+b-1)/2 \rfloor$, by \eqref{eq:4t}, $|B_{t}(\bar{0})| = 1+4\sum_{j=1}^t j = 2t(t+1)+1$. 

Suppose that $(\b)/(\a)$ is a perfect $t$-code in $G_{\a}$. Then by (\ref{eq:deg2}) we have $c^2 + d^2 = 2t(t+1)+1$, and by (\ref{eq:t1}), $|\t \b| \ge 2t+1$ for any $0 \ne \tau \in \Z[i]$.  Since $|i^j \t \b| = |\t \b|$ for any integer $j$, by multiplying $\tau$ by $i, i^2$ or $i^3$ when necessary, we may assume that $\t = f \pm g i$ where $f, g \ge 0$ with $(f, g) \ne (0, 0)$. Note that $\t \b = (cf \mp dg) + (df \pm cg)i$ when $\b = c+di$ and $\t \b = (cf \pm dg) - (df \mp cg)i$ when $\b = c-di$. In both cases, (\ref{eq:t1}) is equivalent to 
\be
\label{eq:pm1}
|cf - dg| + |df + cg| \ge 2t+1,\quad 
|cf + dg| + |df - cg| \ge 2t+1
\ee
for any integers $f, g \ge 0$ with $(f, g) \ne (0, 0)$. 

Assume $c \ge d$ first. Choosing $(f, g) = (1, 1)$ in (\ref{eq:pm1}), we obtain $2c \ge 2t+1$ and so $c \ge t+1$. This together with $c^2 + d^2 = 2t(t+1)+1$ implies $d \le t$. Choosing $(f,g)=(1,0)$ in (\ref{eq:pm1}), we obtain $c+d \ge 2t+1$. If $c+d > 2t+1$, then $2t(t+1)+1 = c^2 + d^2 > ((2t+1)-d)^2 + d^2$, yielding $0 > (d-t)(d-(t+1))$. However, this cannot happen as $d \le t$. Hence $c+d = 2t+1$. Combining this with $c^2 + d^2 = 2t(t+1)+1$, we obtain $cd = t(t+1)$. Therefore the only possibility is that $c = t+1$ and $d=t$.  

Now assume $c < d$. Setting $(f, g) = (1, 1)$ in (\ref{eq:pm1}), we have $2d \ge 2t+1$ and so $d \ge t+1$. This together with $c^2 + d^2 = 2t(t+1)+1$ implies $c \le t$. Choosing $(f,g)=(1,0)$ in (\ref{eq:pm1}), we obtain $c+d \ge 2t+1$. Similar to the argument above, we then obtain $c = t$ and $d = t+1$. 

We conclude the proof by noting that $(t+1) + ti = i(t-(t+1)i)$ and $(t+1) - ti = i^3 (t+(t+1)i)$. 
\qed
\end{proof} 

\begin{remark}
Theorem \ref{thm:gaussian} can be restated as follows: Let $\b = t \pm (t+1)i$ with $t$ a positive integer. Then for any $\a = (x+yi)\b$ or $(x-yi)\b$, where $x, y \ge 0$, $(x, y) \ne (0, 0)$, $G_{\a}$ has $(\b)/(\a)$ as a perfect $t$-code. Moreover, up to isomorphism these are the only cyclotomic graphs $G_{\g}$ with $\b$ dividing $\g$ that admit $(\b)/(\g)$ as a perfect $t$-code in $G_{\g}$.  
\end{remark}

We now move on to the third cyclotomic graphs $EJ_{\a} = G_{3}(\a)$ (see Remark \ref{rem:rho}), where $0 \ne \a = a+b\rho$ and $\rho = (1+\sqrt{-3})/2$. Since $G_{\a} \cong G_{\rho ^j \a}$ for any integer $j$, without loss of generality we may assume $a, b \ge 0$ in $EJ_{\a}$. The size of the ball $B_{t}(\bar{0})$ of radius $t$ around $\bar{0} = 0 + (\a)$ in $EJ_{\a}$ was determined in \cite{FB} for any $t \ge 0$. In particular, it was proved in \cite[Theorem 27]{FB} that, if $a \ge b \ge 0$, then 
\be
\label{eq:6t}
|B_{t}(\bar{0})|=6t,\;\, 0 \le t < (a+b)/2. 
\ee 
Note that this formula is also valid when $0 \le a < b$ as $G_{a+b\rho} \cong G_{b+a\rho}$ (see \cite[Section IV]{FB}). It is known that two elements $\b, \g$ of $\Z[\rho]$ are associates of each other if and only if $\b = \rho^j \g$ for some integer $j$. Since $\Z[\rho]$ is a principal ideal domain and associates generate the same principal ideal, any nonzero ideal of $\Z[\rho]$ is of the form $(c+d\rho)$ or $(c-d\rho)$ for some integers $c, d \ge 0$. 

In \cite[Section IV]{FB}, the \emph{$\rho$-taxicab norm} of $\g \in \Z[\rho]$ was defined as 
$$
|\g|_{\rho} := \min\{|x|+|y|+|z|: \g = x + y\rho + z \rho^2,\ x, y, z \in \Z\}
$$
and the \emph{EJ-norm} of $\bar{\g} = \g + (\a)$ in $EJ_{\a}$ was defined as
$$
\|\bar{\g}\|_{E} := \min\{|\g - \eta \a|_{\rho}: \eta \in \Z[\rho]\}.
$$
Since $\|\bar{\g}_1\|_{E} = \|\bar{\g}_2\|_{E}$ whenever $\g_1 \equiv \g_2\ \mod \a$, $\|\bar{\g}\|_{E}$ is well-defined. It was proved in \cite[Section IV]{FB} (see also \cite{MBG}) that the distance in $EJ_{\a}$ between $\bar{\b}$ and $\bar{\g}$ is given by $\|\bar{\b}-\bar{\g}\|_{E}$. 

The `if' part of the next result was proved in \cite[Theorem 24]{MBG} under the assumption $\gcd(a, b) = 1$. (Note that in \cite[Theorem 24]{MBG} $\b$ has a different form due to the usage of $\omega = (-1+\sqrt{-3})/2$ there.) We now show that the condition $\gcd(a, b) = 1$ can be removed, by using \cite[Theorem 27]{FB} and the argument in the proof of \cite[Theorem 24]{MBG}. Moreover, by using Corollary \ref{cor:ptdom}(b), we prove that the `only if' part is also true. 

\begin{thm}
\label{thm:EJ}
Let $0 \ne \a = a + b \rho \in \Z[\rho]$ (where $a, b \ge 0$), and let $0 \ne \b \in \Z[\rho]$ be such that $N(\a) \ge 7$ and $\b$ divides $\a$. Let $t$ be an integer between $1$ and $\lfloor (a+b-1)/2 \rfloor$. Then $(\b)/(\a)$ is a perfect $t$-code in $EJ_{\a}$ if and only if $\b$ is an associate of $(t+1)+t \rho$ or $t + (t+1) \rho$.  
\end{thm}

\begin{proof}
As noted above, we may assume $0 \ne \b = c \pm d \rho \in \Z[\rho]$, where $c, d \ge 0$. Since $\b$ divides $\a$, we have $(\a) \subseteq (\b)$, $\a = \g \b$ for some $\g \in \Z[i]$, and $N(\b) = c^2 \pm cd + d^2$ divides $N(\a) = a^2 + ab + b^2$. Since $N(\a) \ge 7$, $EJ_{\a}$ has valency $6$ by Theorem \ref{thm:basic}(b). Since $1 \le t \le (a+b-1)/2$, by \eqref{eq:6t}, $|B_{t}(\bar{0})| = 1+6\sum_{j=1}^t j = 3t(t+1)+1$. 

\smallskip
\textsf{Necessity.}~Suppose that $(\b)/(\a)$ is a perfect $t$-code in $EJ_{\a}$. Then by (\ref{eq:deg2}), $c^2 \pm cd + d^2 = 3t(t+1)+1$, and by (\ref{eq:t1}), $|\tau \b| \ge 2t+1$ for every $0 \ne \tau \in \Z[\rho]$. Since $|\rho^j \tau \b| = |\tau \b|$ for any integer $j$, multiplying $\tau$ by an appropriate $\rho^j$ when necessary we may assume $\t = f \pm g \rho$, where $f, g \ge 0$ with $(f, g) \ne (0, 0)$. Note that $\t \b = cf + (df \pm cg)\rho \pm dg\rho^2 = (cf \mp dg) + (df \pm (c+d)g)\rho$ when $\b = c + d\rho$, and $\t \b = cf - (df \mp cg)\rho \mp dg\rho^2 = (cf \pm dg) - (df \mp (c-d)g)\rho$ when $\b = c - d\rho$. 

\smallskip
\textsf{Case 1: $\b = c + d\rho$.}~In this case, $c^2 + cd + d^2 = 3t(t+1)+1$ by (\ref{eq:deg2}), and (\ref{eq:t1}) is equivalent to
\be
\label{eq:pm2}
|cf - dg| + |df + (c+d)g| \ge 2t+1, \quad |cf + dg| + |df - (c+d)g| \ge 2t+1
\ee 
for any integers $f, g \ge 0$ with $(f, g) \ne (0, 0)$. Setting $(f, g) = (1, 0)$, we obtain $c+d \ge 2t+1$. 

Assume $c \ge d$ first. In this case we have $c \ge t+1$ as $c+d \ge 2t+1$. This together with $c^2 + cd + d^2 = 3t(t+1)+1$ implies 
$$
\begin{array}{lll}
d & = & \frac{1}{2} \left(-c+\sqrt{4(3t(t+1)+1)-3c^2}\right) \\
   & \le & \frac{1}{2} \left(-(t+1)+\sqrt{4(3t(t+1)+1)-3(t+1)^2}\right) \\
   & = & t.
\end{array}
$$ 
If $c+d > 2t+1$, then $3t(t+1)+1 = c^2 + cd + d^2 > ((2t+1)-d)^2 + ((2t+1)-d)d + d^2$, yielding $0 > (d-t)(d-(t+1))$. Since this contradicts the fact $d \le t$, we must have $c+d = 2t+1$. This together with $c^2 + cd + d^2 = 3t(t+1)+1$ implies $cd = t(t+1)$. Therefore, $(c, d) = (t+1, t)$. 

Now assume $c < d$. Then $2d > c + d \ge 2t+1$ and so $d \ge t+1$. Similar to the previous paragraph, we then have $c \le t$ and based on this we can further prove that $(c, d) = (t, t+1)$. 

\smallskip
\textsf{Case 2: $\b = c - d\rho$.}~In this case, $c^2 - cd + d^2 = 3t(t+1)+1$ by (\ref{eq:deg2}), and (\ref{eq:t1}) is equivalent to
\be
\label{eq:pm3}
|cf + dg| + |df - (c-d)g| \ge 2t+1,\quad |cf - dg| + |df + (c-d)g| \ge 2t+1
\ee 
for any integers $f, g \ge 0$ with $(f, g) \ne (0, 0)$. 

Assume $c \ge d$ first. Since $c^2 - cd + d^2 = 3t(t+1)+1$, we have
\be
\label{eq:d}
d = \frac{1}{2} \left(c \pm \sqrt{4(3t(t+1)+1)-3c^2}\right).  
\ee
Since $d$ is a real number, we have $3c^2 \le 4(3t(t+1)+1) = 3(2t+1)^2 + 1$, which implies $c \le 2t+1$. On the other hand, setting $(f, g) = (0, 1)$ in \eqref{eq:pm3}, we obtain $c = d + |c-d| \ge 2t+1$. Hence $c = 2t+1$. Plugging this into \eqref{eq:d}, we obtain $d = t+1$ or $t$. Therefore, $(c, d) = (2t+1, t+1)$ or $(2t+1, t)$. 

Next assume $c < d$. Similar to \eqref{eq:pm3}, we have
\be
\label{eq:c}
c = \frac{1}{2} \left(d \pm \sqrt{4(3t(t+1)+1)-3d^2}\right), 
\ee
which implies $d \le 2t+1$. On the other hand, setting $(f, g) = (1, 1)$ in \eqref{eq:pm3}, we obtain $d = |c-d|+|d+(c-d)| \ge 2t+1$. Hence $d = 2t+1$. Plugging this into \eqref{eq:c}, we obtain $c = t+1$ or $t$. Therefore, $(c, d) = (t+1, 2t+1)$ or $(t, 2t+1)$.

It can be verified that $(2t+1) - (t+1)\rho = \rho^5[(t+1)+t\rho]$, $(2t+1) - t\rho = \rho^5[t+(t+1)\rho]$, $(t+1) - (2t+1)\rho = \rho^4 [t+(t+1)\rho]$ and $t - (2t+1)\rho = \rho^4 [(t+1)+t\rho]$. So the ideals $(\b)$ in Case 2 give rise to the same perfect $t$-codes as in Case 1.

\smallskip
\textsf{Sufficiency:}~We use essentially the same argument as in the proof of \cite[Theorem 24]{MBG}, but we do not require $\gcd(a, b) = 1$. Suppose first that $\b = (t+1)+t\rho$ divides $\a$. We aim to prove that $(\b)/(\a)$ is a perfect $t$-code in $EJ_{\a}$. Since $|B_{t}(\bar{0})| = N(\b) = 3t(t+1)+1$, it suffices to prove that the distance $\|\bar{\g}-\bar{\d}\|_{E}$ in $EJ_{\a}$ between any two vertices $\bar{\g},  \bar{\d} \in (\b)/(\a)$ is at least $2t+1$ (see the proof of Lemma \ref{lem:ptdom}), or equivalently, $\|\overline{\g \b}\|_{E} \ge 2t+1$ for any $0 \ne \g \in \Z[\rho]$. Suppose otherwise. Since $\b$ divides $\a$, there exist $0 \ne \eta \in \Z[\rho]$ and integers $x, y, z$ such that $\eta \b = x+y\rho + z \rho^2$ and $\|\overline{\g\b}\|_{E} = |x| + |y| + |z| \le 2t$. Set $\eta = f + g \rho$, where $(f, g) \ne (0, 0)$ are integers. Then $\eta \b = (f(t+1)-gt) + (ft+g(2t+1))\rho$. On the other hand, we have $\eta \b = (x-z)+(y+z)\rho$. Hence $x-z = f(t+1)-gt$, $y+z = ft+g(2t+1)$ and $x+y = f(2t+1)+g(t+1)$. It follows that $|x|+|z| \ge |f(t+1)-gt|$, $|y| + |z| \ge |ft+g(2t+1)|$ and $|x| + |y| \ge |f(2t+1)+g(t+1)|$. Thus, if $|f| < |g|$, then $|y|+|z| \ge |g(2t+1)| - |ft| \ge (|f|+1)(2t+1) - |f|t \ge 2t+1$. Similarly, if $|f| > |g|$, then $|x| + |y| \ge 2t+1$. Moreover, if $f = g \ne 0$ then $|x|+|y| \ge 2t+1$, and if $f = -g \ne 0$ then $|x|+|z| \ge 2t+1$. In any case, we have $|x|+|y|+|z| \ge 2t+1$, a contradiction. Therefore, the distance in $EJ_{\a}$ between any two distinct vertices of $(\b)/(\a)$ is at least $2t+1$. Consequently, the balls $B_{t}(\bar{\g})$, $\bar{\g} \in (\b)/(\a)$ are pairwise disjoint. However, there are $N(\a)/N(\b)$ such balls and each of them has size $N(\b)$. Therefore, these balls form a partition of the vertex set $\Z[\rho]/(\a)$ of $EJ_{\a}$. That is, $(\b)/(\a)$ is a perfect $t$-code in $EJ_{\a}$.  

It can be verified that, for $\b = t+(t+1)\rho$ and $\eta = f+g\rho$, we have $\eta \b = (ft-g(t+1)) + (f(t+1)+g(2t+1))\rho$. Using this and a similar argument as above, one can show that $(\b)/(\a)$ is a perfect $t$-code in $EJ_{\a}$ provided that $\b$ divides $\a$.
\qed
\end{proof}

\section{Circulant cyclotomic graphs}
\label{sec:cir-cyclo}

In this section we present a family of circulant cyclotomic graphs of valency twice an odd prime, namely $2p$-valent first kind Frobenius circulants. We give a classification of all such graphs in Theorem \ref{thm:2p} and then prove that they are indeed cyclotomic in Theorem \ref{cor:cir-cyclo}. 

Since $\Z[\z_m]$ is a $\Z$-module with integral basis $1, \z_m, \ldots, \z_m^{\phi(m)-1}$, we may write 
\be
\label{eq:phi}
\z_m^i = \sum_{j=0}^{\phi(m)-1}c_{ij}\z_m^j,\; 0 \le i \le m-1, 
\ee
where all $c_{ij} \in \Z$ are determined by $\z_m$. Note that, for $0 \le i \le \phi(m)-1$, we have $c_{ii} = 1$ and $c_{ij} = 0$ when $i\ne j$. The next result gives a construction of circulant cyclotomic graphs. 

\begin{lem}
\label{thm:cir-cyclo}
Let $m \ge 2$ and $n \ge 3$ be odd integers, and let $c_{ij}$ be defined by (\ref{eq:phi}). Suppose that $a$ is a positive integer such that 
\be
\label{eq:a}
a^i \equiv \sum_{j=0}^{\phi(m)-1}c_{ij} a^j\; \mod n,\;\, \phi(m) \le i \le m-1  
\ee
and $a^{m} \equiv 1~\mod n$ but $a^{i} \not \equiv \pm 1~\mod n$ for $1 \le i \le m-1$. 
Then $\Cay(\Z_n, \la [-a] \ra) \cong G_m(A_{m,n,a})$, where
\be
\label{eq:Amna}
A_{m,n,a} := \left\{\sum_{i=0}^{\phi(m)-1} a_{i} \z_m^{i} \in \Z[\z_m]: \sum_{i=0}^{\phi(m)-1} a_{i} a^{i} \equiv 0~\mod n\right\}.
\ee 
\end{lem}

\begin{proof}
Define 
\be
\label{eq:f}
f\left(\sum_{i=0}^{\phi(m)-1} a_{i} \z_m^{i}\right) = \sum_{i=0}^{\phi(m)-1} a_{i} a^{i}\;\, \mod n,\;\, a_i \in \Z.
\ee
Since $1, \z_m, \ldots, \z_m^{\phi(m)-1}$ is an integral basis for the $\Z$-module $\Z[\z_m]$, $f$ is a well-defined mapping from $\Z[\z_m]$ to $\Z_n$. Obviously, $f$ is surjective. Using (\ref{eq:phi})-(\ref{eq:f}), one can verify that $f\left(\sum_{i=0}^{k} a_{i} \z_m^{i}\right) = \sum_{i=0}^{k} a_{i} a^{i}$, $0 \le k \le n-1$. This can be easily extended to arbitrary $k$, that is, for any $k \ge 0$,  
\be
\label{eq:hom}
f\left(\sum_{i=0}^{k} a_{i} \z_m^{i}\right) = \sum_{i=0}^{k} a_{i} a^{i}\;\, \mod n.
\ee 

We claim that $f$ is a ring homomorphism from $\Z[\z_m]$ to $\Z_n$. In fact, for $\a = \sum_{i=0}^{m-1} a_{i} \z_m^{i} \in \Z[\z_m]$ and $\b = \sum_{i=0}^{m-1} b_{i} \z_m^{i} \in \Z[\z_m]$, it is evident that $f(\a+\b) = f(\a) + f(\b)$. Since $\z_m^m = 1$, we have 
$$
\a \b = \sum_{k=0}^{m-1} \left(\sum_{i+j=k} a_i b_j\right) \z_m^{k} + 
\sum_{k=m}^{2(m-1)} \left(\sum_{i+j=k} a_i b_j\right) \z_m^{k-m}.
$$
Thus, by (\ref{eq:hom}) and the assumption $a^{m} \equiv 1~\mod n$,  
\bean
f(\a \b) & = & \sum_{k=0}^{m-1} \left(\sum_{i+j=k} a_i b_j\right) a^{k} + 
\sum_{k=m}^{2(m-1)} \left(\sum_{i+j=k} a_i b_j\right) a^{k-m}\; \mod n \\
& = & \left(\sum_{i=0}^{m-1} a_{i} a^{i}\right) \left(\sum_{i=0}^{m-1} b_{i} a^{i}\right)\; \mod n\\
& = & f(\a) f(\b).
\eean
Therefore, $f$ is a surjective homomorphism from $\Z[\z_m]$ to $\Z_n$.  

The kernel of $f$ is exactly $A = A_{m,n,a}$ as defined in \eqref{eq:Amna}. By the homomorphism theorem for rings, we have $\Z[\z_m]/A \cong \Z_n$ and $\bar{f}(x + A) := f(x)$, $x \in \Z[\z_m]$ defines the corresponding isomorphism from $\Z[\z_m]/A$ to $\Z_n$. Since $f(\z_m^{i}) = a^i\; \mod n$ by (\ref{eq:hom}), $\bar{f}$ maps $\pm (\z_m^i + A)$ to $\pm [a^i]$, $0 \le i \le m-1$. Since $m$ is odd, it follows that the subset $E_{m}/A$ of $\Z[\z_m]/A$ defined in (\ref{eq:pma}) with respect to $A$ above is the pre-image of $\la [-a] \ra = \{\pm [a^i]: 0 \le i \le m-1\} \le \Z_n^*$ under $\bar{f}$. Since $n$ is odd and $a^{m} \equiv 1~\mod n$, $n$ is not a divisor of $2a^i$ for $0 \le i \le m-1$. This together with the assumption $a^{i} \not \equiv \pm 1~\mod n$, $1 \le i \le m-1$ implies that $\la [-a] \ra$ has order $2m$. Therefore, $E_{m}/A$ has size $2m$ and $\bar{f}$ gives a bijection from $E_{m}/A$ to $\la [-a] \ra$. In other words, $G_{m}(A)$ has valency $2m$. It is readily seen that $\bar{f}$ gives rise to an isomorphism from $G_{m}(A)$ to $\Cay(\Z_n, \la [-a] \ra)$. 
\qed
\end{proof}

\begin{lem} 
\label{semiregular}
(\cite[Lemma 4]{TZ})
Let $n \ge 3$ be an integer. A subgroup $H$ of $\Z_n^*$ is semiregular on $\Z_n \setminus \{[0]\}$ if and only if $[h-1] \in \Z_n^*$ for all $[h] \in H \setminus \{[1]\}$.
\end{lem}
 
\begin{thm}
\label{thm:2p}
Let $p$ be an odd prime and $n \ge 2p+1$ an integer. Then a $2p$-valent circulant $\Cay(\Z_n, S)$ with $[1] \in S$ is a first kind Frobenius graph with cyclic kernel if and only if $n \equiv 1~\mod 2p$ and $S = \la [a] \ra$ for some positive integer $a$ such that $a^{p}+1 \equiv 0~\mod n$ and $\gcd(a^i \pm 1, n) = 1$ for $1 \le i \le p-1$. Moreover, in this case $\Cay(\Z_n, \la [a] \ra)$ is a $\Z_n \rtimes \la [a] \ra$-arc transitive first kind $\Z_n \rtimes \la [a] \ra$-Frobenius circulant.   
\end{thm}

\begin{proof}
Let $\Cay(\Z_n, S)$ be a first kind Frobenius circulant with order $n$ such that $[1] \in S$ and the kernel of the underlying Frobenius group is $\Z_n$. Then there exists a subgroup $H$ of $\Z_n^*$ such that $|H| = 2p$, $\Z_n \rtimes H$ is a Frobenius group and $\Cay(\Z_n, S)$ is a first kind $\Z_n \rtimes H$-Frobenius circulant. Thus $H$ is semiregular on $\Z_n \setminus \{[0]\}$, and in particular $n \equiv 1~\mod 2p$. Moreover, $S$ is an $H$-orbit on $\Z_n$ and hence $H$ is regular on $S$. Since $[1] \in S$, it follows that $S = H$. Since $H$ is Abelian with $|H| = 2p$, it is a cyclic group of order $2p$, as an Abelian group of order $2p$ must be cyclic. So we may assume $H = \la [a] \ra = \{[a^i]: 0 \le i \le 2p-1\}$, where $[a]$ is an element of $\Z_n^*$ with order $2p$. Since $[1] \in S$ and $S$ is closed under taking negative elements, we have $-[1] \in S = H$ and so there exists $i$ with $2 \le i \le 2p-1$ such that $[a^i] = -[1]$ (note that $[a] \ne -[1]$ as $[a]$ has order $2p > 2$ in $\Z_n^*$). Thus $[a^{2i}] = [1]$ and so $2p$ divides $2i$. Since $p$ is a prime, we have $i = p$ and therefore $a^p + 1 \equiv 0~\mod n$ (so that $H = \{\pm [1], \pm [a], \pm [a^2], \ldots, \pm [a^{p-1}]\}$). Since $H$ is semiregular on $\Z_n \setminus \{[0]\}$, by Lemma \ref{semiregular}, the integers $a^i \pm 1$ are all coprime to $n$ for $1 \le i \le p-1$. 

Conversely, if $n \equiv 1~\mod 2p$ and $a$ is a positive integer such that $a^{p}+1 \equiv 0~\mod n$ and $a^i \pm 1$, $1 \le i \le p-1$ are coprime to $n$, then $H = \la [a] \ra \le \Z_n^*$ is semiregular on $\Z_n \setminus \{[0]\}$ with order $|H| = 2p$. Therefore, $\Z_n \rtimes H$ is a Frobenius group and $\Cay(\Z_n, \la [a] \ra)$ is a first kind $\Z_n \rtimes H$-Frobenius graph. Moreover, $\Cay(\Z_n, \la [a] \ra)$ is $\Z_n \rtimes H$-arc-transitive by \cite[Lemma 2.1]{Z}.
\qed
\end{proof} 

\begin{remark}
Since $a^p + 1 = (a+1)\sum_{i=0}^{p-1} (-1)^i a^{i}$ and $a^2-1 = (a-1)(a+1)$, the conditions in Theorem \ref{thm:2p} are equivalent to that $a^{p-1} \equiv \sum_{i=0}^{p-2} (-1)^{i+1} a^{i}~\mod n$ and $\gcd(a^i \pm 1, n) = 1$ for $2 \le i \le p-1$. Thus each $[u] \in \Z_n$ can be expressed as $[u] = [\sum_{i=0}^{p-2} u_{i} a^{i}]$ for some integers $u_0, u_1, \ldots, u_{p-2}$. Obviously this representation is not unique and without loss of generality we may assume $u_i \ge 0$ for $0 \le i \le p-2$. The neighbours of $[u]$ are $[u] + [a^j]$, $0 \le j \le 2p-1$, and $H = \la [a] \ra$ cyclically `rotates' the `directions' $[a^j]$ at $[u]$ in the obvious way. From a geometric point of view this determines a cyclic permutation of the edges incident with $[u]$ and thus defines an embedding of $\Cay(\Z_n, \la [a] \ra)$ on a closed orientable surface as a balanced regular Cayley map (see \cite[Corollary 2.9]{TZ2}). Note that $\Cay(\Z_n, \la [a] \ra)$ is a rotational Cayley graph. 
\end{remark}

\begin{thm}
\label{cor:cir-cyclo}
Let $p$ be an odd prime and $n \ge 2p+1$ an integer with $n \equiv 1~\mod 2p$. Then the first kind Frobenius circulant $\Cay(\Z_n, \la [a] \ra)$ in Theorem \ref{thm:2p} is isomorphic to $G_{p}(A_{p, n, -a})$. 
\end{thm}

\begin{proof}
We have $\phi(p) = p-1$, $\z_p^{p-1} = -\sum_{j=0}^{p-2} \z_p^j$, and $A_{p,n,-a} = \{\sum_{i=0}^{p-2} a_{i} \z_p^{i} \in \Z[\z_p]: \sum_{i=0}^{p-2} a_{i} (-a)^{i} \equiv 0~\mod n\}$. Since $a^p + 1 \equiv 0~\mod n$ and $a+1$ is coprime to $n$, we have $(-a)^{p-1} \equiv -\sum_{j=0}^{p-2} (-a)^j~\mod n$, which means that $-a$ satisfies the condition (\ref{eq:a}). So $\Cay(\Z_n, \la [a] \ra) \cong G_{p}(A_{p, n, -a})$ by Lemma \ref{thm:cir-cyclo}. 
\qed
\end{proof} 

\begin{remark} 
In general, $A_{m,n,a}$ defined in \eqref{eq:Amna} is not necessarily a principal ideal. However, it must be a principal ideal if $m$ is one of the following integers:
$$
3, 4, 5, 7, 8, 9, 11, 12, 13, 15, 16, 17, 19, 20, 21, 24, 25, 27, 28, 32, 33, 35, 36, 40, 44, 45, 48, 60, 84.
$$
This is because there are precisely 29 cyclotomic fields $\Q(\z_m)$ with $\Z[\z_m]$ a principal ideal domain and they are given by these integers $m$ \cite{MM}. 
Thus, by Theorem \ref{cor:cir-cyclo}, we know that for $p = 3, 5, 7, 11, 13, 17, 19$, $\Cay(\Z_n, \la [a] \ra)$ in Theorem \ref{thm:2p} is isomorphic to $G_{p}(\a)$ for some $0 \ne \a \in \Z[\z_p]$. It would be interesting to investigate when the converse of this statement is true (see \cite[Theorem 5(b)]{TZ3} in the case when $p = 3$). 
\end{remark} 
 
\begin{problem}
\label{prob}
Let $t \ge 1$ be an integer. For $p = 5, 7, 11, 13, 17, 19$, find necessary and sufficient conditions for $(\b)/(\a)$ to be a perfect $t$-code in $G_{p}(\a)$ (or $G^*_{p}(\a)$), where $0 \ne \a, \b \in \Z[\z_p]$ such that $\b$ divides $\a$. 
\end{problem}

In view of Corollary \ref{cor:ptdom}, the first key step towards this problem may be to acquire detailed knowledge of the distance in $G_{p}(\a)$ (or $G^*_{p}(\a)$) and the size of the $t$-neighbourhood of a vertex in the graph. 

In the case when $p = 3$, Theorem \ref{cor:cir-cyclo} asserts that, for any odd integer $n \ge 7$ and positive integer $a$ such that $a^2 - a + 1 \equiv 0~\mod n$ and $a^2 \pm 1$ is coprime to $n$, the 6-valent first kind Frobenius circulant 
$$
TL_{n}(a, a-1, 1) := \Cay(\Z_n, \la [a] \ra)
$$ 
is isomorphic to the Eisenstein-Jacobi graph $EJ_{\a} = G_{3}(A_{3, n, -a})$ (see Remark \ref{rem:rho}), a result noticed in \cite[Theorem 5(a)]{TZ3} (with more details), where $A_{3, n, -a} = \{c+d\rho \in \Z[\rho]: c + d a \equiv 0~\mod n\} = (\a)$ for some $0 \ne \a \in \Z[\rho]$ as $\Z[\rho]$ is an Euclidean domain.  

\begin{figure}[ht]
\centering
\includegraphics*[height=11.0cm]{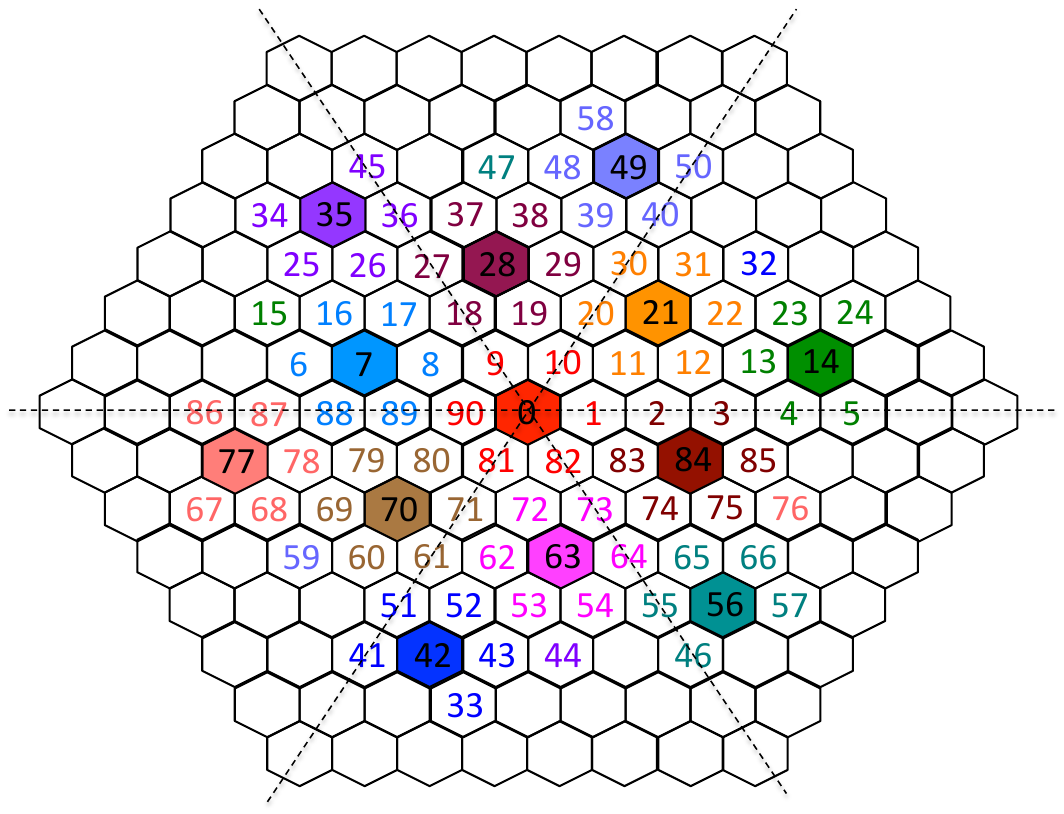}
\caption{\small A perfect $1$-code in $EJ_{1+9\rho} \cong TL_{91}(10, 9, 1)$.} 
\label{fig:hex91}
\end{figure}  

We finish this paper by the following example to illustrate Theorems \ref{thm:EJ} and \ref{cor:cir-cyclo}. 

\begin{ex}
\label{ex:91}
Let $a = 10$ and $n = a^2 - a + 1 = 91$. Then by Theorem \ref{cor:cir-cyclo} (see also \cite[Example 2]{TZ3}) $TL_{91}(10, 9, 1)$ is isomorphic to $EJ_{\a}$ for some $0 \ne \a \in \Z[\rho]$. In fact, by \cite[Theorem 5(a)]{TZ3}, $\a = 1+9\rho$ and $f: x+10y~\mod 91 \mapsto x+y\rho~\mod \a$ is an isomorphism from $TL_{91}(10, 9, 1)$ to $EJ_{1+9\rho}$, where $x$ and $y$ are integers. By Theorem \ref{thm:EJ}, the only perfect $t$-codes in $EJ_{1+9\rho}$ of the form $(\b)/(1+9\rho)$ are given by $\b = (t+1) + t \rho$, $t + (t+1)\rho$ with $\b$ dividing $1+9\rho$, where $1 \le t \le 4$. One can see that, for $t = 2, 3, 4$, $N(\b)$ is not a divisor of $N(1+9\rho) = 91$ and so $\b$ does not divide $1+9\rho$. Moreover, $1+2\rho$ does not divide $1+9\rho$ whilst $1+9\rho = (2+\rho)(4-\rho)$. Therefore, the only perfect code in $EJ_{1+9\rho}$ of the form $(\b)/(1+9\rho)$ is $(2+\rho)/(1+9\rho)$, which is a perfect $1$-code with size $N(1+9\rho)/N(2+\rho) = 13$. 

It can be verified that $(2+\rho)/(1+9\rho) = \{j(1+2\rho)~\mod (1+9\rho): 0 \le j \le 12\}$. Since $f^{-1}: j(1+2\rho)~\mod (1+9\rho) \mapsto 21j~\mod 91$, $0 \le j \le 12$, we may view $(2+\rho)/(1+9\rho)$ as the perfect $1$-code $C := \{0, 21, 42, 63, 84, 14, 35, 56, 77, 7, 28, 49, 70\}$ ($\mod 91$) in $TL_{91}(10, 9, 1)$. Following \cite[Section 5]{TZ3}, we can represent this graph by its minimum distance diagram as shown in Figure \ref{fig:hex91} (the area with numbers), where each vertex is adjacent to the six vertices in the neighbouring cells. By the discussion in \cite[Section 5]{TZ3}, the whole plane can be tessellated by copies of this minimum distance diagram. The 13 coloured vertices (numbers) in Figure \ref{fig:hex91} constitute the perfect $1$-code $C$, and the ball of radius one centred at each coloured vertex consists of the coloured vertex itself and its six neighbours. For example, the ball of radius one centred at $84$ is $\{84, 3, 2, 83, 74, 75, 85\}$, and that centred at $42$ is $\{42, 52, 51, 41, 32, 33, 43\}$. Equivalently, we can label the hexagonal cells by the elements of $\Z[\rho]/(1+9\rho)$, say, $21 = 1+2 \cdot 10$ can be replaced by $1+2\rho$, $78 = 8 + 7 \cdot 10$ by $8 + 7 \rho$, and so on. 
\end{ex}

\section*{Acknowledgements}
   
The author was supported by the Australian Research Council (FT110100629). He thanks Alex Ghitza for helpful discussions on number theory and He Huang for critical comments on earlier versions of this paper.


{\small

\begin{thebibliography}{99}

  


\bibitem{Bannai}
E.~Bannai, On perfect codes in the Hamming schemes $H(n,q)$ with $q$ arbitrary, 
{\em J. Combin. Theory (A)} {\bf 23} (1977), 52--67. 

\bibitem{BCH}
J.-C. Bermond, F.~Comellas and D.~F.~Hsu, Distributed loop computer networks: a survey, {\em J. Parallel Dist. Comput.} {\bf 24} (1995), 2--10. 

\bibitem{Biggs}
N.~Biggs, Perfect codes in graphs, {\em J. Combin. Theory (B)} {\bf 15} (1973), 289--296.

\bibitem{DS}
I.~Dejter and O.~Serra, Efficient dominating sets in Cayley graphs, {\em Discrete Appl. Math.} {\bf 129}  (2003), 319--328.

\bibitem{Del}
P.~Delsarte, An algebraic approach to the association schemes of coding
theory, {\em Philips Res. Repts Suppl.} {\bf 10} (1973), 1--97.

\bibitem{DSLW16}
Y-P. Deng, Y-Q. Sun, Q. Liu and H.-C. Wang, Efficient dominating sets in circulant graphs, \emph{Discrete Math.} \textbf{340} (2017), 1503--1507.
  
\bibitem{Dixon-Mortimer}
J.~D.~Dixon and B.~Mortimer, {\em Permutation Groups}, Springer,
New York, 1996.
  
\bibitem{E87}
G. Etienne, Perfect codes and regular partitions in graphs and groups, {\em European J. Combin.} {\bf 8} (1987), no. 2, 139--144.

\bibitem{E}
T.~Etzion, On the nonexistence of perfect codes in the Johnson scheme, {\em SIAM J. Discrete Math.} {\bf 9} (1996), 201--209. 

\bibitem{FG}
Y. Fan and Y. Gao, Codes over algebraic integer rings of cyclotomic fields, \emph{IEEE Trans. Inform. Theory} \textbf{50} (2004), 194--200. 

\bibitem{FLP}
X.~G.~Fang, C.~H.~Li and C.~E.~Praeger, On orbital regular graphs and Frobenius graphs, \emph{Discrete Math.} \textbf{182} (1998), 85--99.

\bibitem{FHZ16}
R. Feng, H. Huang and S. Zhou, Perfect codes in circulant graphs, \emph{Discrete Math.} \textbf{340} (2017), 1522--1527.


\bibitem{FB}
M.~Flahive and B.~Bose, The topology of Gaussian and Eisenstein-Jacobi interconnection networks, {\em IEEE Trans. Parallel Distrib. Syst.} {\bf 21} (2010), 1132--1142.

\bibitem{Heden}
O.~Heden, A survey of perfect codes, {\em Adv. Math. Commun.} {\bf 2} (2008), 223--247.


\bibitem{HMP}
M.-C.~Heydemann, N.~Marlin and S.~P\'{e}renes, Complete rotations in Cayley graphs, {\em Europ.~J.~Combin.} {\bf 22} (2001), 179--196.

\bibitem{HMS}
M.-C.~Heydemann, J.-C.~Meyer and D.~Sotteau, On forwarding indices of networks, {\em Discrete Applied Math.} {\bf 23} (1989),
103-123.

\bibitem{HBZ18}
H. Huang, B. Xia and S. Zhou, Perfect codes in Cayley graphs, \emph{SIAM J. Discrete Math.} \textbf{32} (2018), 548--559.
 
\bibitem{Huber94}
K.~Huber, Codes over Gaussian integers, {\em IEEE Trans. Inform. Theory} {\bf 40} (1994), 207--216.
 
\bibitem{IR}
K.~Ireland and M.~Rosen, {\em A Classical Introduction to Modern Number Theory}, 2nd ed., Springer-Verlag, New York, 1990.

\bibitem{JS}
R.~Jajcay and J.~\v{S}ir\'{a}\v{n}, Skew-morphism of regular
Cayley maps, {\em Discrete Math.} {\bf 244} (2002), 167--179.

\bibitem{Kra}
J.~Kratochv\'{i}l, Perfect codes over graphs, 
{\em J. Combin. Theory (B)} {\bf 40} (1986), 224--228.
 
\bibitem{Lee}
J.~Lee, Independent perfect domination sets in Cayley graphs, 
{\em J. Graph Theory} {\bf 37} (2001), 213--219. 

\bibitem{MBCSG10} 
C.~Mart\'{i}nez, R.~Beivide, C.~Camarero, E.~Stafford and E. M. Gabidulin, Quotients of Gaussian graphs and their application to perfect codes, {\em J. Symbolic Comput.} {\bf 45} (2010), 813--824.

\bibitem{MBG}
C.~Mart\'{i}nez, R.~Beivide and E.~Gabidulin, Perfect codes for metrics induced by circulant graphs, {\em IEEE Trans. Inform. Theory} {\bf 53} (2007), 3042--3052.

\bibitem{MBG09}
C. Mart\'{i}nez, R. Beivide and E. Gabidulin, Perfect codes from Cayley graphs over Lipschitz integers, {\em IEEE Trans. Inform. Theory} {\bf 55} (2009), no. 8, 3552--3562.

\bibitem{MBSMG}
C.~Mart\'{i}nez, R.~Beivide, E.~Stafford, M.~Moret\'{o} and E. M.~Gabidulin, Modeling toroidal networks with the Gaussian integers, {\em IEEE Trans. Computers} {\bf 57} (2008), no. 8, 1046--1056.


\bibitem{MM}
J. M. Masley and H. L. Montgomery, Cyclotomic fields with unique factorization, {\em J. Reine Angew. Math.}{\bf 286/287} (1976), 248--256. 

\bibitem{OPR07}
N. Obradovi\'{c}, J. Peters and G. Ru\v{z}i\'{c}, Efficient domination in circulant graphs with two chord lengths, {\em Inform. Process. Lett.} {\bf 102} (2007), no. 6, 253--258. 
 

\bibitem{KM13}
K. Reji Kumar and G. MacGillivray, Efficient domination in circulant graphs, {\em Discrete Math.} {\bf 313} (2013), no. 6, 767--771.

\bibitem{SS}
M.~\v{S}koviera and J.~\v{S}ir\'{a}\v{n}, Regular maps from Cayley
graphs, Part I: balanced Cayley maps, {\em Discrete Math.} {\bf
109} (1992), 265--276.

\bibitem{S}
P.~Sol\'{e}, The edge-forwarding index of orbital regular graphs,
{\em Discrete Math.} {\bf 130} (1994), 171--176.


\bibitem{T04}
S. Terada, Perfect codes in $\SL(2, 2^f)$, {\em European J. Combin.} {\bf 25} (2004), no. 7, 1077--1085.

\bibitem{TZ}
A.~Thomson and S.~Zhou, Frobenius circulant graphs of valency four, {\em J. Austral. Math. Soc.} {\bf 85} (2008), 269--282. 

\bibitem{TZ1}
A.~Thomson and S.~Zhou, Gossiping and routing in undirected triple-loop networks, {\em Networks} {\bf 55} (2010), 341--349.

\bibitem{TZ2}
A.~Thomson and S.~Zhou, Rotational circulant graphs, {\em Discrete Applied Math.} {\bf 162} (2014), 296--305.

\bibitem{TZ3}
A.~Thomson and S.~Zhou, Frobenius circulant graphs of valency six, Eisenstein-Jacobi networks, and hexagonal meshes, {\em Europ. J. Comb.} {\bf 38} (2014), 61--78.

\bibitem{vanLint}
J.~H.~van Lint, A survey of perfect codes, {\em Rocky Mountain J. Math.} {\bf 5} (1975), 199--224.

\bibitem{Wash}
L.~C.~Washington, {\em Introduction to Cyclotomic Fields}, 2nd ed., Springer, New York, 1997.   

\bibitem{Ze}
J.~\v{Z}erovnik, Perfect codes in direct products of cycles - a complete characterization, {\em Adv. in Appl. Math.} {\bf 41} (2008), 197--205.

\bibitem{Z}
S.~Zhou, A class of arc-transitive Cayley graphs as models for interconnection networks, {\em SIAM J. Discrete Math.} {\bf 23} (2009), 694--714.

\bibitem{Z1}
S.~Zhou, On 4-valent Frobenius circulant graphs, {\em Disc. Math. and Theoretical Comp. Sci.} {\bf 14} (2) (2012), 173--188. 

\end{thebibliography}
}
\end{document}